\documentclass[11pt,twoside]{amsart}
\usepackage{amsmath, amsthm, amscd, amsfonts, amssymb, graphicx, color}
\usepackage[bookmarksnumbered, colorlinks, plainpages]{hyperref}

\newtheorem{thm}{Theorem}[section]

\newtheorem{defn}[thm]{Definition}
\newtheorem{question}[thm]{Question}
\newtheorem{example}[thm]{Example}
\newtheorem{rem}[thm]{\bf{Remark}}

\numberwithin{equation}{section}
\def\pn{\par\noindent}

\begin{document}



\title{On the possible volume of $\mu$-$(v,k,t)$ trades}
\author{Saeedeh Rashidi  and Nasrin Soltankhah$^*$}

\thanks{{\scriptsize
\hskip -0.4 true cm MSC(2010): Primary 05B30; Secondary05B05.
$*$Corresponding author
}}
\maketitle


\begin{abstract}
A $\mu$-way $(v,k,t)$ $trade$ of volume $m$ consists of $\mu$
disjoint collections $T_1$, $T_2, \dots T_{\mu}$, each of $m$
blocks, such that for every $t$-subset of $v$-set $V$ the number of
blocks containing this t-subset is the same in each $T_i\ (1\leq
i\leq \mu)$. In other words any pair of collections $\{T_i,T_j\}$,
$1\leq i<j \leq \mu$ is a $(v,k,t)$ trade of volume $m$.
 In this paper we investigate the existence of $\mu$-way $(v,k,t)$ trades  and also we prove
 the existence of: (i)~3-way $(v,k,1)$ trades (Steiner
trades) of each volume $m,m\geq2$. (ii) 3-way $(v,k,2)$ trades of
each volume $m,m\geq6$ except possibly $m=7$. We establish the
non-existence of 3-way $(v,3,2)$ trade of volume 7. It is shown that
the volume of a 3-way $(v,k,2)$  Steiner trade is at least $2k$ for
$k\geq4$. Also the spectrum of 3-way $(v,k,2)$  Steiner trades for
$k=3$ and 4 are specified.
\end{abstract}

\vskip 0.2 true cm


%



\section{\bf Introduction}
\vskip 0.4 true cm

Given a set of $v$ treatments $V$, let $k$ and $t$ be two positive integers such that   $t<k<v$.
A $(v,k,t)$ $trade$ $T=\{T_1,T_2\}$ of volume $m$  consists of two disjoint collections $T_1$ and $T_2$,
each of containing $m$ $k$-subsets of  $V$, called $blocks$,
such that every $t$-subset of $V$ is contained in the same number of blocks in $T_1$ and $T_2$.
A  $(v,k,t)$ trade is called $(v,k,t)$ $Steiner$ trade if any t-subset of $V$ occurs in at most once in $T_1(T_2)$.\\
A $t-(v,k,\lambda)\  design$ $({V},{B})$ is a collection of blocks such that each $t-$subset of $V$ is contained in
$\lambda$ blocks.\\
When $m=0$ the trade is said to be $void$. A $(v,k,t)$ trade (design) is also a $(v,k,t')$ trade (design), for all
$t'$ with $0<t'<t$.
In a $(v,k,t)$ trade, both collections of
blocks must cover the same set of elements. This set of elements is called the $foundation$ of the trade and is denoted by $found(T)$. \\
A $2-(v,3,1)$ design is called a $Steiner\ triple\ system$ of order $v$ and is often denoted by $STS(v)$.
It is well known that a $STS(v)$ exists if and only if $v\equiv 1,3\ (\rm mod\ 6)$.\\
A $Kirkman\ triple\ system$ of order $v$ that is often denoted by $KTS(v)$ is a Steiner triple system
of order $v$ $({V},{B})$ together with a partition $R$ of the set of triples $B$
into subsets $R_1,R_2,\dots,R_n$ called parallel classes such that each $R_i\ (i=1,2,\dots,n)$ is a partition of $V$.\\
A $partial\  triple\  system$ (PTS) is a pair $(V, P)$ where $V$ is a finite nonempty n-set and $P$ is a collection of 3-subsets of $V$,
called blocks (or triples), such that every pair of distinct elements of $V$ is contained in at most one block of $P$.\\
Two partial triple systems $(V, P_1)$ and $(V,P_2)$ are said to be disjoint and mutually balanced (DMB) if:\\
(i) $P_1\cap P_2=\phi$.\\
(ii) any given pair of distinct elements of $V$ is contained in a block of $P_1$ if and only if it is contained in a block of $P_2$.\\
Milici and Quattrocchi (1986) used what is now known as Steiner trades and named them, DMB (disjoint and mutually balanced).
The concept of trade was first introduced in 1960s by Hedayat \cite{17}. Hedayat and Li applied the method of trade-off
and trades for building BIBDs with repeated blocks (1979-1980).
Papers by Hwang~\cite{2}, Mahmoodian and Soltankhah~\cite{4}, and Asgari and Soltankhah~\cite{12}
deal with the existence and non-existence of $(v,k,t)$ trades.
Concept of trade was introduced for BIBDs first and then it was used in the Latin squares with Latin trade title
(see \cite{10}) and in the Graph theory with $G$-trade title (see \cite{14}).\\
The definition of trades can be generated, and here we introduce $\mu$-way trades $(\mu\geq 2)$ as follows:
\begin{defn}
A $\mu$-way $(v,k,t)$ $trade$ of volume $m$ consists of $\mu$ disjoint collections $T_1$, $T_2, \dots T_{\mu}$,
each of $m$ blocks,
such that for every $t$-subset of $v$-set $V$ the number of blocks containing this t-subset is the same in each
$T_i\ (1\leq i\leq \mu)$.
In other words any pair of collections $\{T_i,T_j\}$, $1\leq i<j \leq \mu$ is a $(v,k,t)$ trade of volume $m$.
\end{defn}
\begin{defn}
A $\mu$-way $(v,k,t)$ trade is called  $\mu$-way $(v,k,t)$ $Steiner\  trade$ if any t-subset of $\rm{found(T)}$
occurs at most once in $T_1\ (T_j,\  j\geq2)$.
\end{defn}
\begin{example}~\label{11}
\rm The following trades are  3-way $(8,3,2)$ Steiner trade and 3-way $(11,3,2)$ Steiner trade of volume 8 and 13,
respectively:\\

\begin{tabular}{c|c|c}
$T_1$&$T_2$&$T_3$\\
\hline
$1,2,3$&$1,2,4$&$1,2,7$\\
$1,4,7$&$1,3,8$&$1,3,5$\\
$1,5,8$&$1,5,7$&$1,4,8$\\
$2,4,8$&$2,3,7$&$2,4,6$\\
$2,6,7$&$2,6,8$&$2,3,8$\\
$3,5,7$&$4,6,7$&$3,6,7$\\
$3,6,8$&$4,5,8$&$4,5,7$\\
$4,5,6$&$3,5,6$&$5,6,8$
\end{tabular}
$\hspace{2cm}$
\begin{tabular}{c|c|c}
$T_1$&$T_2$&$T_3$\\
\hline
$1,2,11$&$2,3,11$&$1,3,11$\\
$  3,10,11$&$1,10,11$&$2,10,11$\\
$1,3,7$&$1,2,8$&$1,2,9$\\
$  1,10,9$&$1,3,5$&$1,10,8$\\
$1,5,8$&$1,7,9$&$1,5,7$\\
$ 2,3,6$&$2,10,9$&$2,3,4$\\
$ 2,10,8$&$2,4,6$&$2,6,8$\\
$ 2,4,9$&$3,10,6$&$3,10,7$\\
$ 3,4,5$&$3,4,7$&$3,5,6$\\
$ 6,8,4$&$6,5,8$&$7,9,4$\\
$ 6,10,5$&$7,5,10$&$9,5,10$\\
$ 7,9,5$&$9,4,5$&$8,4,5$\\
$ 7,10,4$&$8,4,10$&$6,4,10$
\end{tabular}
\end{example}
Trades are also intimately connected with the so-called $intersection$ problem for combinatorial structures.
This basically asks, given two combinatorial
structures with the same parameters, and based on the same underlying set,
such as a pair of block designs or a pair of latin rectangles, in how many ways
may they intersect? So for two block designs, how many common blocks may there be?
Of course, removing a set of $m$ blocks from a design and replacing them with
a distinct set of $m$ blocks which nevertheless still make the whole collection of blocks a
design with the same parameters, is utilising a trade of volume $m$
to yield two designs with $m$ blocks different, and so a known number of blocks in common.\\
The $intersection\  problem$ has also been considered for more than just $pairs$ of combinatorial
structures; the intersection of $\mu$ combinatorial structures with $\mu>2$ was dealt with in, for example,~\cite{5}
for three Steiner triple systems and~\cite{11} for three latin squares.
These correspond in the same manner to $\mu$-way trades in the corresponding combinatorial structure.\\
So it is clear that if there exist three $t-(v,k,\lambda)$ designs $(V,B)$ which intersect in the same set of $m$ blocks,
and which differ in the remaining blocks then we obtain
a 3-way $(v',k,t)$ trade of volume $b_v-m$ where $b_v=|{B}|$.
Conversely if $D=(V,B)$ be a $t-(v,k,\lambda)$ design and $T=\{T_1,T_2,T_3\}$ be a
3-way $(v,k,t)$ trade of volume $m$. If $T_1\subseteq B,$ we say that
$D$ contains the trade $T$, and if we replace $T_i \ (i=2,3)$ with $T_1$,
then we obtain new designs $D_i=(D-T_1)\cup T_i$ which are denoted
by $D_i=D+T_i$ with same parameters of $D$, and $|D_i\cap D|=|D_i\cap D_j|=b_v-m\ (2\leq i,j\leq 3)$.
 If there is not a 3-way $(v',k,t)$ trade of volume $m$, then there does not exist three designs with intersection number $b_v-m$.\\
It is important to understand the structure of $\mu$-way trades and conditions for their existence and non-existence.
Here, the following question is of interest.
\begin{question}
For a given $\mu$, what is the set of all possible volume sizes (the ``volume spectrum'') of a $\mu$-way $(v,k,t)$ trade?
\end{question}
We now introduce some notations. Let  $\mathcal{S}_{\mu}(t,k)$
($\mathcal{S}_{\mu s }(t,k)$) denote the set of all possible volume
sizes of a $\mu$-way $(v,k,t)$ trade ($\mu$-way $(v,k,t)$ Steiner
trade).

This question has been answered for $\mu =2$ until now as follows:\\
\begin{enumerate}
\item \cite{2}~$\mathcal{S}_{2}(2,k)=\mathbb{N}\setminus\{1,2,3,5\}.$\\
\item \cite{3}~$\mathcal{S}_{2s}(2,3)=\mathbb{N}\setminus\{1,2,3,5\}.$\\
\item \cite{1}~$\mathcal{S}_{2s}(2,4)=\mathbb{N}\setminus\{1,2,3,4,5,7\}.$\\
\item \cite{9}~$\mathcal{S}_{2s}(2,5)=\mathbb{N}\setminus\{1,2,3,4,5,6,7,9,11\}.$\\
\item \cite{9}~$\mathcal{S}_{2s}(2,6)=\mathbb{N}\setminus\{1,2,3,4,5,6,7,8,9,11,13\}.$\\
\item \cite{9}~$\text{If}~0<m<2k-2~\text{or}~m=2k-1~\text{then}~m\not\in \mathcal{S}_{2s}(2,k).$\\
\item \cite{9}~$\text{If}~m=0, m\geq 3k-3~\text{or}~m~\text{is even and}~2k-2~\leq~m~\leq~3k-~4\\ \text{then}~m\in \mathcal{S}_{2s}(2,k).$\\
\item \cite{13}~$2k+1\in  \mathcal{S}_{2s}(2,k)~\text{precisely when}~k\in\{3,4,7\}.$\\
\item \cite{15}~$\text{If}~m~\text{is odd and}~2k+3\leq m\leq 3k-4,~\text{then}\\ \mathcal{S}_{2s}(2,k)~\text{does not contain}~m~\text{for}~k~\geq~7.$\\
\item \cite{8}~$\mathcal{S}_{2s}(3,4)=\mathbb{N}\setminus\{1,2,3,4,5,6,7,9,10,11,13\}.$
\end{enumerate}

In this paper for $\mu=3$, we investigate this question and our
results include the following.\\

\textbf{Main results:}\\
 $\ $
(1) $\mathcal{S}_{3}(1,k)=\mathcal{S}_{3s}(1,k)=\mathbb{N}\setminus\{1\},\ k\geq2$.\\
(2) $\mathcal{S}_{3}(2,3)=\mathbb{N}\setminus\{1,2,3,4,5,7\}$.\\
(3) $\mathcal{S}_{3}(2,k)\setminus―{7\}=\mathbb{N}\setminus\{1,2,3,4,5\}$.\\
(4) $\mathcal{S}_{3s}(2,3)=\mathbb{N}\setminus\{1,2,3,4,5,7\}$.\\
(5) $\mathcal{S}_{3s}(2,4)=\mathbb{N}\setminus\{1,2,3,4,5,6,7\}$.\\
(6) $S_{3s}(2,k)\subseteq \mathbb{N}\setminus\{1,2,\dots,2k-1―}$.\\


\section{\bf Preliminary results}
\vskip 0.4 true cm

We start this section with some notation and useful results.
Let $T=\{T_1,\dots,T_{\mu}\}$ be a $\mu$-way $(v,k,t)$ trade of volume $m$, and $x,y\in \rm{found(T)}$.\\
Then the number of blocks in $T_i\ (1\leq i\leq \mu)$ which contains $x$ is denoted by $r_x$ and the number of blocks
containing $\{x ,y\}$ is denoted by $\lambda_{xy}$. The set of blocks in $T_i\ (1\leq i\leq \mu)$ which contains $x\in \rm{found(T)}$
is denoted by $T_{ix}\ (1\leq i\leq \mu)$ and the set of remaining blocks by $T'_{ix}\ (1\leq i\leq \mu)$.\\
By applying a result in \cite{2}, we see that if $r_x<m$, then $T_x=\{T_{1x},\dots,T_{\mu x}\}$ is a $\mu$-way $(v,k,t-1)$ trade of volume
$r_x,$ and furthermore $T'_x=\{T'_{1x},\dots,T'_{\mu x}―}$  is a $\mu$-way $(v-1,k,t-1)$ trade of volume
$m-r_x$. If we remove $x$ from the blocks of $T_x$, then the result will be a $\mu$-way $(v-1,k-1,t-1)$ trade which is called derived
trade of $T$.\\
It is easy to show that if $T$ is a Steiner trade then its derived trade is also a Steiner trade.\\
If $T=\{T_1,\dots,T_{\mu}\}$ and $T^{*}=\{T^{*}_1,\dots,T^{*}_{\mu}\}$ are two $\mu$-way $(v,k,t)$ trades.
Then we define $T+T^{*}=\{T_1\cup T_1^{*},\dots, T_{\mu}\cup T_{\mu}^{*}\}$. It is easy to see that $T+T^{*}$ is a
$\mu$-way $(v,k,t)$ trade. If $T$ and $T^{*}$ are Steiner trades and $\rm {found(T)}\cap \rm{found(T^*)}=\phi$, then $T+T^{*}$
is also a Steiner trade.
\begin{defn}
Let $T=\{T_1,T_2,\dots,T_{\mu}\}$ be a $\mu$-way $(v,k,t)$ Steiner trade. We say $T$
is t-solely balanced if $T_{i}$ and $T_{j}$ $(1\leq i < j\leq \mu)$ contain no common $(t+1)-$subset.
\end{defn}
Following theorem will be used repeatedly in the sequel.
\begin{thm}~\label{1.9}
 (i) Let $T=\{T_1,T_2,\dots,T_{\mu}\}$ be a $\mu$-way $(v,k,t)$ trade of volume $m$. Then, based on $T$,
a $\mu$-way $(v+\mu,k+1,t+1)$ trade $T^*$ of volume $\mu m$ can be constructed.\\
 (ii) If $T$ is $t$-solely balanced, then $T^*$ is a Steiner trade.
\end{thm}
\begin{proof}
(i) Let $x_1,x_2$ and $x_{\mu}$ be $\mu$ new elements. Then we can construct the blocks of $T^*~=\{T^*_1,T^*_2,\dots,T^*_{\mu}\}$ as follows.
\begin{center}
\begin{tabular}{c|c|c|c}
$T^*_{1}$& $T^*_{2}$&$\dots$ & $T^*_{\mu}$\\
\hline
$x_1T_1$ & $x_1T_2$& $\dots$ & $x_1T_{\mu}$\\
$x_2T_2$& $x_2T_3$ & $\dots$& $x_2T_{1}$\\
$x_3T_3$& $x_3T_4$ & $\dots$& $x_3T_2$\\
$\vdots$& $\vdots$ & $\vdots$ &  $\vdots$\\
$x_{\mu}T_{\mu}$& $x_{\mu}T_1$ & $\dots$& $x_{\mu}T_{\mu-1}$\\
\end{tabular}
\end{center}
Clearly $T^*$ is a $\mu$-way $(v+\mu,k+1,t+1)$ trade of volume $\mu m$.\\
(ii) It is obvious.
\end{proof}
In the next example, we show  the existence of a 3-way $(v,3,2)$ Steiner trade of volume 6 from
a 3-way $(v,2,1)$ Steiner trade of volume 2.
\begin{example}~\label{2.3}
Let $T=\{T_1,T_2,T_{3}\}$ be the 3-way $(v,2,1)$ Steiner trade of volume 2.
\begin{center}
\begin{tabular}{c|c|c}
$T_{1}$&$T_{2}$&$T_{3}$\\
\hline
$12$&$13$&$14$\\
$34$&$24$&$23$\\
\end{tabular}
\end{center}
Now we can construct $T^*~=\{T^*_1,T^*_2,T^*_{3}\}$ by the method of the previous Theorem.
\begin{center}
\begin{tabular}{c|c|c}
$T^*_{1}$&$T^*_{2}$&$T^*_{3}$\\
\hline
$x12$&$x13$&$x14$\\
$x34$&$x24$&$x23$\\
$13z$&$12y$&$12z$\\
$24z$&$34y$&$34z$\\
$14y$&$14z$&$13y$\\
$23y$&$23z$&$24y$\\
\end{tabular}
\end{center}
\end{example}

\begin{rem}~\label{23}
\rm{The 3-way $(v,3,2)$ Steiner trade of volume 6 is unique. This trade is isomorphic to the
3-way $(7,3,2)$ Steiner trade of volume 6  which is constructed in Example~\ref{2.3}.\\
Let $T$ be a 3-way $(v,3,2)$ Steiner trade of volume 6. First assume that, for each $x\in \rm found(T)$, $r_x>2$.
So $x$ must appear at least 3 times in $T_1$. Let the first block of $T_{1x}$ be $x12$. So 1 and 2 must appear
at least two times in $T'_{1x}$, since $r_1,r_2\geq3$. Hence
$x$, 1 and 2 should each appear twice more in different blocks which contradicts the Steiner property of $T$.
So there exists $x\in\rm{ found(T)}$ such that $r_x=2$.
We know $T_x\setminus\{x\}$ is a 3-way $(v,2,1)$ Steiner trade. Therefore $T_x$ can be expressed as:
\begin{center}
\begin{tabular}{c|c|c}
$T_{1x}$&$T_{2x}$&$T_{3x}$\\
\hline
$x12$&$x13$&$x14$\\
$x34$&$x24$&$x23$\\
\end{tabular}
\end{center}
Thus the pairs $13,\ 24,\ 14$ and 23 must appear in distinct blocks of $T_1$. Since $T$ is  a 3-way $(v,3,2)$ Steiner trade.
Therefore a 3-way $(v,3,2)$ Steiner trade of volume 6 has the following structure.
\begin{center}
\begin{tabular}{c|c|c}
$T_{1}$&$T_{2}$&$T_{3}$\\
\hline
$x12$&$x13$&$x14$\\
$x34$&$x24$&$x23$\\
$13z$&$12y$&$12z$\\
$24z$&$34y$&$34z$\\
$14y$&$14z$&$13y$\\
$23y$&$23z$&$24y$\\
\end{tabular}
\end{center}}
\end{rem}

\begin{thm}~\label{101}
$\mathcal{S}_{\mu}(2,k)\subseteq\mathbb{ N}\setminus\{1,2,3,4,5\}, \ k\geq3$.
\end{thm}
\begin{proof}
We know $\mathcal{S}_{2}(2,k)=\mathbb{N}\setminus\{1,2,3,5\}$ (see~\cite{2}).
So $\mathcal{S}_{\mu}(2,k)\subseteq\mathbb{ N}\setminus\{1,2,3,5\}$.\\
The $(v,k,2)$ trade of volume 4 has unique structure (see~\cite{2}). If there exists a 3-way $(v,k,2)$ trade $T=\{T_1, T_2, T_3\}$ of
volume 4, then $(T_1,T_2)$, $(T_2,T_3)$ and $(T_1, T_3)$ are three $(v,k,2)$ trades of volume 4 and it is a contradiction,
because the structure of $(v,k,2)$ trade of volume 4 is unique.
\end{proof}
%
\section{\bf 3-way Steiner trades }
%
In this section we characterize $\mathcal{S}_{3s}(1,k)$, $\mathcal{S}_{3s}(2,3)$ and $\mathcal{S}_{3s}(2,4)$.
First, we state some
of the results in \cite{6} which are needed in the sequel.\\
Let $D(v,k)$ be the maximum number of $STS(v)$s that can be constructed on a set with cardinality
$v$ such that any two $STS(v)$s intersect
exactly in the same $k$ blocks.
\begin{thm}~\label{301}
\cite{6} $D(v,b_v-7)=2$ for every $v\geq7$; $v\ne9$.
\end{thm}
\begin{thm}~\label{302}
\cite{30} Any partial Steiner triple system of order $v$ can be embedded in a
Steiner triple system of order $w$ if $w\equiv 1, 3\ (\rm mod\ 6)$ and $w\geq 2v+1.$
\end{thm}
\begin{thm}~\label{7}
$7\not \in\mathcal{S}_{3s}(2,3)$.
\end{thm}
\begin{proof}
Let $T=\{T_1,T_2,T_3\}$ be a 3-way $(v,3,2)$  Steiner trade of volume 7. It is obvious that $T_1$  is a
partial Steiner triple system. So by Theorem~\ref{302}, $T_1$
can be embeded in a $STS(v')=D$, where $v'\geq 2|\rm{found(T)}|+1\geq13$. Then $D$, $D'=D+T_2$ and $D''=D+T_3$ are
three $STS(v)$s which intersect in the same set of $b_v-7$ blocks. But this is impossible by Theorem~\ref{301}.
\end{proof}
\begin{thm}~\label{2.1}
$\mathcal{S}_{3s}(1,k)=\mathbb{N}\setminus\{1\},\ k\geq2$.
\end{thm}
\begin{proof}
We know the complete graph $K_{2m}$ has $2m-1$ disjoint 1-factors. If we take three
1-factors $F_1,\ F_2$ and $F_3$ as $T_1,\ T_2$ and $T_3$ respectively,
then $T=\{T_1,T_2,T_3\}$ is a 3-way $(2m,2,1)$ trade of volume $m$.\\
For $k\geq3,$ let $T$ be a 3-way $(v,2,1)$ Steiner trade of volume $m$
and $A$ be a $(k-2)m$-set disjoint from $\rm{found(T)}$.
Set a partition of $A$ to $(k-2)$ subsets $A_1,\dots,A_m$. Then by adding $A_i\ (1\leq i\leq m)$
to the ith block of $T$, we obtain a 3-way $(v,k,1)$ Steiner trade.
\end{proof}
\begin{example}
a 3-way $(4,2,1)$ Steiner trade of volume 2.
\begin{center}
\begin{tabular}{c|c|c}
$T_{1}$& $T_{2}$ & $T_3$\\
\hline
$13$ & $14$ & $12$\\
$24$& $23$ & $34$\\
\end{tabular}
\end{center}
a 3-way $(8,4,1)$ Steiner trade of volume 2.
\end{example}
\begin{center}
\begin{tabular}{c|c|c}
$T_{1}$& $T_{2}$ & $T_3$\\
\hline
$x_1x_213$ & $x_1x_214$ & $x_1x_212$\\
$x_3x_424$& $x_3x_423$ & $x_3x_434$\\
\end{tabular}
\end{center}

\begin{thm}~\label{102}
$\mathcal{S}_{3s}(2,3)=\mathbb{N}\setminus\{1,2,3,4,5,7\}.$
\end{thm}
\begin{proof}
By Theorem~\ref{1.9} (ii) and Theorem~\ref{2.1} there exists a 3-way $(v,3,2)$ Steiner trade of volume $3m\ (m\geq2)$.
Note that the 3-way $(2m,2,1)$ Steiner trades of volume $m$ constructed in Theorem~\ref{2.1} are 1-solely balanced.
The existence of a 3-way $(v,3,2)$ Steiner trade of volumes  $3m+1$ and $3m+2$,
can be proved by using two following recursive relations:\\
(i) $3m+1=3(m-3)+10\ \ \ \ m-3\geq2$;\\
 (ii) $3m+2=3(m-2)+8\ \ \ \ m-2\geq2$.\\
These constructions, together with 3-way $(v,3,2)$  Steiner trades of volumes: 8, 10, 11
and 13 suffice to prove the existence.\\
We can see a 3-way $(v,3,2)$  Steiner trade of volume 8 and 13 in Example~\ref{11}.\\
Consider three $STS(v)$s intersecting in $b_v-m$ blocks,  where $b_v=\frac{v(v-1)}{6}$.
The remaining set of blocks form a 3-way $(v',3,2)$ Steiner trade
of volume $m$. We know that  there exist three $STS(9)$s which intersect in $b_9-11=12-11=1$
block and three $STS(v)$s which intersect in
$b_v-10$ blocks for $v\geq19$  (see \cite{5}). So $ \{10,11\}\subseteq \mathcal{S}_{3s}(2,3) $.\\
The non-existence of Steiner trades of volumes $1,2,3,4,5$ and 7 can be concluded from Theorems~\ref{101} and~\ref{7}.
\end{proof}

\begin{thm}~\label{35}
$\mathcal{S}_{3s}(2,k)\subseteq\mathbb{ N}\setminus\{1,2,...,2k-1\}$ for $k\geq 4$.
\end{thm}
\begin{proof}
Let $T=\{T_1,T_2,T_3\}$ be a 3-way $(v,k,2)$  Steiner trade of volume $m$.
Let for each $x\in \rm{found(T)}$  $r_x\geq 3$, and $a_1,\dots,a_k$ be a block
in $T_1$. Corresponding to each $a_i$, there exist two other blocks in $T_1$,
which contain $a_i\ (1\leq i\leq k)$ but not $a_j$ $(j\ne i)$ (Since $T$ is a Steiner trade).
$T_1$ must contain at least $2k+1$ blocks.\\
Now let there exists $x\in\rm{found(T)}$ such that $r_x=2,$ then $T_x$ is
a 3-way $(v,k,1)$ Steiner trade. So $(T_{1x},T_{2x})$ has the following form from~\cite{2}.
\begin{center}
\begin{tabular}{c|c}
$T_{1x}$& $T_{2x}$\\
\hline
$S_{1}S_{3}S_{5}$ & $S_{1}S_{4}S_{5}$\\
$S_{2}S_{4}S_{5}$& $ S_{2}S_{3}S_{5}$\\
\end{tabular}
\end{center}
with $S_{i}\subseteq V$ for $i=1,\cdots,5$. $|S_{1}|=|S_{2}|\geq 1$, $|S_{3}|=|S_{4}|\geq 1$,
$S_{i}\cap S_{j}=\phi$ for all $i\ne j$, and $|S_{1}|+|S_{3}|+|S_{5}|=k.$\\
Since $T$ is a 3-way $(v,k,2)$  Steiner trade and $r_x=2$, therefore $S_5=\{x\}$.
Without loss of generality, let\\
$S_1S_3=a_2a_3a_4\dots a_k$ and  $S_2S_4=b_2b_3b_4\dots b_k$. So there exists $i$ such that  \\
$S_1S_4=a_2\dots a_ib_{i+1}\dots b_k$ and  $S_2S_3=b_2\dots b_ia_{i+1}\dots a_k$.
Then corresponding to each pair $a_pb_q$ and $b_pa_q$ in $T_2$, $2\leq p\leq i$
and $i+1\leq q\leq k$, there must exist $2(i-1)(k-i)$ blocks in $T_1$.\\
We know there does not exist a repetitive block in $T_3$. So $a_2$ must appear in $T_3$ with
some $b_j, j\notin\{i+1,\dots,k\}$ or with some $a_j, j\in\{i+1,\dots,k\}$
(one block of $T_{3x}$ contains $a_2$ and $b_{i+1}\dots b_k$).
In the first case we have at least one block for $a_2b_j$ in $T_1$. if the second case happen,
we have $k-i$ blocks for $a_jb_r\ r\in\{i+1,\dots,k\}$. Therefore
in two cases, there exists at least another block in $T_1$. We have same situation for $b_2$. Then we have:
$$|T_1|\geq2+2(i-1)(k-i)+2\geq2k-2+2=2k.$$
So the volume of 3-way $(v,k,2)$  Steiner trade is at least $2k$.
\end{proof}
The 3-way $(v,k,1)$ Steiner trades $(k\geq3),$ which were constructed in
Theorem~\ref{2.1}, are not 1-solely balanced.
But for $k=3$ by using the idea of Kirkman triple systems, in the following
theorem we introduce 3-way $(v,3,1)$ Steiner trade
1-solely balanced.
\begin{thm}~\label{6}
There exists a 3-way $(v,3,1)$ Steiner trade 1-solely balanced of volume $m\ (m\geq{3})$.
\end{thm}
\begin{proof}
We know, there exists a $KTS(v)$ if and only if $v\equiv 3\ \rm{(mod \ 6)}$~\cite{18}.
For $m=2k+1$, consider  a $KTS(3m)$. Let $P_1,P_2,P_3$, be three parallel classes of $KTS(3m)$.
We can construct  a 3-way $(v,3,1)$ Steiner trade 1-solely balanced of volume $m$ as follows.
\begin{center}
\begin{tabular}{c|c|c}
$T_{1}$& $T_{2}$ & $T_{3}$ \\
\hline
$P_1$ &  $P_2$ & $P_3$\\
\end{tabular}
\end{center}
For $m=2k$, consider two 3-way $(v,3,1)$ Steiner trades
1-solely balanced $T$ and $T'$ of odd volumes with disjoint foundations,
then $T+T'$ is a 3-way $(v,3,1)$ Steiner trade 1-solely balanced
of volume $m$, except $m=4,$ which we handle below.
\begin{center}
\begin{tabular}{c|c|c}
$T_1$&$T_2$&$T_3$\\
\hline
$123$&$147$&$158$\\
$456$&$25a$&$24c$\\
$789$&$8b6$&$7b3$\\
$abc$&$39c$&$69a$
\end{tabular}
\end{center}
\end{proof}
\begin{thm}~\label{12}
$\mathcal{S}_{3s}(2,4)=\mathbb{N}\setminus\{1,2,3,4,5,6,7\}$.
\end{thm}
\begin{proof}
By Theorem~\ref{35}, $\mathcal{S}_{3s}(2,4)\subseteq\mathbb{N}\setminus\{1,2,3,4,5,6,7\}$.\\
By Theorem~\ref{6} and Theorem~\ref{1.9} (ii), there exists a 3-way $(v,4,2)$ Steiner trade of volume $3m$ for $m ―geq3$.\\
The existence of a 3-way $(v,4,2)$ Steiner trade of volumes  $3m+1$ and $3m+2$,
can be proved by using two following recursive relations:\\
$3m+1=3(m-3)+10\ \ \ \ m-3\geq3$;\\
$3m+2=3(m-2)+8\ \ \ \ m-2\geq3$.\\
These constructions, together with 3-way $(v,4,2)$ Steiner trades of volumes:
8, 10, 11, 13, 14, 16 suffice to prove the existence.\\
We handle volumes $m=8, 10, 11, 13, 14$ and 16 (see appendix).
\end{proof}
\begin{example}
\rm{In this example we can construct a 3-way $(9,3,1)$  trade of volume $3$ from a KTS(9).
Then we can obtain a 3-way $(12,4,2)$ Steiner trade of volume $9$ from it.\\
KTS(9):
\begin{center}
\begin{tabular}{c|c|c|c}
$P_1$&$P_2$&$P_3$&$P_4$\\
\hline
$123$&$147$&$159$&$168$\\
$456$&$258$&$267$&$249$\\
$789$&$369$&$348$&$357$
\end{tabular}
\end{center}
the 3-way $(9,3,1)$ trade of volume 3:
\begin{center}
\begin{tabular}{c|c|c}
$T_1$&$T_2$&$T_3$\\
\hline
$123$&$147$&$159$\\
$456$&$258$&$267$\\
$789$&$369$&$348$
\end{tabular}
\end{center}
the 3-way $(12,4,2)$ Steiner trade of volume 9:
\begin{center}
\begin{tabular}{c|c|c}
$T_1$&$T_2$&$T_3$\\
\hline
$x123$&$x147$&$x159$\\
$x456$&$x258$&$x267$\\
$x789$&$x369$&$x348$\\
$y147$&$y159$&$y123$\\
$y258$&$y267$&$y456$\\
$y369$&$y348$&$y789$\\
$z159$&$z123$&$z147$\\
$z267$&$z456$&$z258$\\
$z348$&$z789$&$z369$
\end{tabular}
\end{center}}
\end{example}
%

\section{\bf 3-way $(v,k,2)$ trades}
%
In the previous section we see that
$\mathcal{S}_{3s}(1,k)=\mathbb{N}\setminus\{1\}$ for $k\geq2$. So
$\mathcal{S}_{3}(1,k)=\mathbb{N}\setminus\{1\}$ for $k\geq2$. In
this section we investigate the spectrums $\mathcal{S}_{3}(2,3)$ and
$\mathcal{S}_{3}(2,k)$.
\begin{thm}~\label{10}
If there exists a 3-way $(v,3,2)$ trade of volume 7, then it is a 3-way $(v,3,2)$ Steiner trade.
\end{thm}
\begin{proof}
Let $T=\{T_1,T_2,T_3\}$ be a 3-way $(v,3,2)$ trade of volume 7. We prove that there does not exist any pair $x,y\in\rm found(T)$
with $\lambda_{xy}\geq2$.\\
First, Suppose that
$\lambda_{xy}\geq3$.
\begin{center}
\begin{tabular}{c|c|c}
$T_{1}$& $T_{2}$ & $T_3$\\
\hline
$xyz_1$ & $xyz_4$ & $xyz_7$\\
$xyz_2$& $xyz_5$ & $xyz_8$\\
$xyz_3$& $xyz_6$ & $xyz_9$\\
$---$&$---$&$---$\\
$---$&$---$&$---$\\
$---$&$---$&$---$\\
$---$&$---$&$---$
\end{tabular}
\end{center}
The pairs $xz_i$ for $i=4\dots 9$ must appear in the blocks of $T_1$. So the element $x$ must appear three times more in $T_1$ and therefore $r_{x}\geq 6$.
But $r_x\ne 6,7$ for all $x\in\rm found(T)$. Because $T_x$ and $T'_x$ are trades of volume $r_x$ and $m-r_x$. We know that there does not exist any trade of volume one.\\
If $\lambda_{xy}=2$ then $T$ has the following form.
\begin{center}
\begin{tabular}{c|c|c}
$T_{1}$& $T_{2}$ & $T_3$\\
\hline
$xyz_1$ & $xyz_3$ & $xyz_6$\\
$xyz_2$& $xyz_4$ & $xyz_5$\\
$---$&$---$&$---$\\
$---$&$---$&$---$\\
$---$&$---$&$---$\\
$---$&$---$&$---$\\
$---$&$---$&$---$
\end{tabular}
\end{center}
The pairs $xz_i$ for $i=3,\dots,6$ must appear in the blocks of $T_1$. So $r_x,r_y\geq4$.
\begin{center}
\begin{tabular}{c|c|c}
$T_{1}$& $T_{2}$ & $T_3$\\
\hline
$xyz_1$ & $xyz_3$ & $xyz_6$\\
$xyz_2$& $xyz_4$ & $xyz_5$\\
$x--$&$x--$&$x--$\\
$x--$&$x--$&$x--$\\
$y--$&$y--$&$y--$\\
$y--$&$y--$&$y--$\\
$---$&$---$&$---$
\end{tabular}
\end{center}
It is obvious that the 3th and 4th blocks of $T_1$, also the 4th and 5th blocks of $T_1$ contain the elements $z_i$ for $i=3,\dots,6$.
Hence the 7th block of $T_1$ contains $z_1$ and $z_2$. Because the order of each element is at least two.\\
Now there exist an empty place in the last block of $T_1$.
By the previous reason, there does not exist any new element in this place. If one of the elements $x, y$ and $z_i$, $i=1,\dots,6$ appears in this place (Name it $w$).
Then the pair $z_1w$ must appear in the blocks of $T_2(T_3)$ and it is impossible.
\end{proof}
\begin{thm}~\label{103}
 $\mathcal{S}_{3}(2,3)=\mathbb{N}\setminus\{1,2,3,4,5,7\}$.
\end{thm}
\begin{proof}
This is concluded from Theorems~\ref{10},~\ref{101}, and~\ref{102}.
\end{proof}
\begin{thm}
 $\mathcal{S}_{3}(2,k)$ contains $\mathbb{N}\setminus\{1,2,3,4,5\}$, except possibly 7.
\end{thm}
\begin{proof}
We have a 3-way $(v,3,2)$ trade of volume $m,\ m\in \mathbb{N}\setminus\{1,2,3,4,5,7\}$ from Theorem~\ref{103}.
Let $A$ be a $(k-3)m$-set disjoint from $\rm{found(T)}$.
Set a partition of $A$ to $(k-3)$ subsets $A_1,\dots,A_m$. Then by adding $A_i\ (1\leq i\leq m)$
to the ith block of $T$, we obtain a 3-way $(v,k,2)$ trade of volume $m,\ m\in \mathbb{N}\setminus\{1,2,3,4,5,7\}$.
The non-existence of 3-way $(v,k,2)$ trades of volume $m,\ m\in\{1,2,3,4,5\}$ is concluded from Theorem~\ref{101}.
\end{proof}
%
\section{\bf Appendix}
The following trades are necessary in the proof of Theorem~\ref{12}.\\
$m=8:$
\begin{tabular}{c|c|c}
$T_1$&$T_2$&$T_3$\\
\hline
$124a$&$125b$&$124b$\\
$1568$&$1468$&$156c$\\
$17bc$&$17ac$&$17a8$\\
$235b$&$234a$&$235a$\\
$346c$&$356c$&$3468$\\
$378a$&$378b$&$37cb$\\
$489b$&$589a$&$4c9a$\\
$59ac$&$49bc$&$59b8$
\end{tabular}
$m=10:$
\begin{tabular}{c|c|c}
$T_1$&$T_2$&$T_3$\\
\hline
$0139$&$0238$&$089c$\\
$028c$&$091c$&$0123$\\
$124a$&$987a$&$824a$\\
$17bc$&$84bc$&$87b3$\\
$235b$&$935b$&$295b$\\
$2679$&$9642$&$267c$\\
$346c$&$376c$&$9463$\\
$378a$&$341a$&$971a$\\
$489b$&$127b$&$41cb$\\
$59ac$&$52ac$&$5ca3$
\end{tabular}\\
\\
\\

$m=11:$
\begin{tabular}{c|c|c}
$T_1$&$T_2$&$T_3$\\
\hline
$028c$&$025c$&$0286$\\
$0457$&$0468$&$045b$\\
$06ab$&$07ab$&$07ac$\\
$1568$&$1675$&$1578$\\
$17bc$&$18bc$&$1bc6$\\
$235b$&$236b$&$235c$\\
$2679$&$2789$&$27b9$\\
$346c$&$347c$&$3476$\\
$378a$&$385a$&$3b8a$\\
$489b$&$459b$&$89c4$\\
$59ac$&$69ac$&$596a$
\end{tabular}
$m=13:$
\begin{tabular}{c|c|c}
$T_1$&$T_2$&$T_3$\\
\hline
$0139$&$149c$&$0739$\\
$028c$&$248c$&$328c$\\
$0457$&$04b7$&$3451$\\
$06ab$&$46a5$&$36ab$\\
$124a$&$18a0$&$724b$\\
$1568$&$1b62$&$7568$\\
$17bc$&$157c$&$17ac$\\
$235b$&$835b$&$205a$\\
$2679$&$8679$&$2691$\\
$346c$&$306c$&$046c$\\
$378a$&$372a$&$018b$\\
$489b$&$0295$&$489a$\\
$59ac$&$b9ac$&$59bc$
\end{tabular}\\
\\
\\

$m=14:$
\begin{tabular}{c|c|c}
$T_1$&$T_2$&$T_3$\\
\hline
$0456$&$1456$&$2456$\\
$28ad$&$08ed$&$18ad$\\
$37be$&$37ba$&$37fe$\\
$19cf$&$29cf$&$09cb$\\
$0789$&$1789$&$2789$\\
$15bd$&$25bd$&$05fd$\\
$24ce$&$04ca$&$14ce$\\
$36af$&$36ef$&$36ab$\\
$0abc$&$1abc$&$2afc$\\
$68e1$&$268a$&$068e$\\
$257f$&$057f$&$157b$\\
$0def$&$1dfa$&$2deb$\\
$147a$&$247e$&$047a$\\
$269b$&$069b$&$169f$
\end{tabular}
$m=16:$
\begin{tabular}{c|c|c}
$T_1$&$T_2$&$T_3$\\
\hline
$0456$&$1456$&$0856$\\
$28ad$&$38ad$&$2bad$\\
$37be$&$07be$&$37fe$\\
$19cf$&$29cf$&$19c4$\\
$0789$&$1789$&$07b9$\\
$15bd$&$25bd$&$15fd$\\
$24ce$&$34ce$&$28ce$\\
$36af$&$06af$&$36a4$\\
$0abc$&$1abc$&$0afc$\\
$68e1$&$268e$&$16be$\\
$257f$&$357f$&$2574$\\
$349d$&$049d$&$38d9$\\
$0def$&$1def$&$0de4$\\
$147a$&$247a$&$187a$\\
$269b$&$369b$&$269f$\\
$358c$&$058c$&$35bc$
\end{tabular}

\vskip 0.4 true cm


\bigskip
\bigskip


{\footnotesize \pn{\bf Saeedeh Rashidi}\; \\ {Department of
Mathematics}, {University
of Alzahra, P.O.Box 19834,} {Tehran, Iran}\\
{\tt Email: s.rashidi@alzahra.ac.ir}\\

{\footnotesize \pn{\bf Nasrin Soltankhah}\; \\ {Department of
Mathematics}, {University
of Alzahra, P.O.Box 19834,} {Tehran, Iran}\\
{\tt Email: soltan@alzahra.ac.ir}\\
\end{document}